\documentclass[journal, onecolumn]{IEEEtran}

\RequirePackage[OT1]{fontenc}
\RequirePackage[numbers,sort]{natbib}
\RequirePackage[colorlinks,citecolor=blue,urlcolor=blue]{hyperref}

\usepackage[ansinew]{inputenc}
\usepackage{graphicx}
\usepackage{textcomp}
\usepackage{amsfonts}
\usepackage{amsmath}
\usepackage{amssymb}
\usepackage{latexsym}
\usepackage[format=hang, justification=RaggedRight]{caption}
\usepackage{array}
\usepackage{float}
\usepackage{amsthm}
\usepackage{calc}
\usepackage{perpage}
\usepackage{textcomp}
\usepackage{verbatim}
\usepackage{marvosym}
\usepackage{hieroglf}
\usepackage{multirow}

\usepackage[english]{babel}
\usepackage{hyperref}
\DeclareMathOperator*{\argmin}{arg\,min}

\def\RR{\mathbb R}
\begin{document}

\title{The Group Square-Root Lasso:\\
Theoretical Properties and Fast Algorithms}

\author{Florentina Bunea, Johannes Lederer, and Yiyuan She
\thanks{F. Bunea is with the Department
of Statistical Science at Cornell University and is supported in part by NSF grant DMS-10-07444, fb238@cornell.edu.}
\thanks{J. Lederer is with the Seminar for Statistics at ETH Z\"urich and acknowledges partial financial support as member of the German-Swiss Research Group FOR916 (Statistical   Regularization and Qualitative Constraints) with grant number 20PA20E-134495/1, johanneslederer@mail.de.}
\thanks{Y. She is with the Department of Statistics at Florida State University and is supported in part by NSF grant CCF-1116447, yshe@stat.fsu.edu.}}

\markboth{The Group Square-Root Lasso}{}

\maketitle

\begin{abstract}
We introduce and study the Group Square-Root Lasso (GSRL)  method for
estimation in high dimensional sparse regression models with group
structure. The new estimator minimizes the square root of the residual
sum of squares plus a penalty term proportional to the sum of the
Euclidean norms of groups of the regression parameter vector. The net
advantage of the method over the existing Group Lasso (GL)-type
procedures consists in the form of the proportionality factor used in
the penalty term, which for GSRL is independent of the variance of the
error terms. This is of crucial importance in models with more
parameters than the sample size, when estimating the variance of the
noise becomes as difficult as the original problem.  We show that the
GSRL estimator adapts to the unknown sparsity of the regression
vector, and has the same optimal estimation and prediction accuracy as
the GL estimators, under the same minimal conditions on the
model. This extends the results recently established for the
Square-Root Lasso, for sparse regression without group
structure. Moreover, as a new type of result for Square-Root Lasso
methods, with or without groups, we study correct pattern recovery,
and show that it can be achieved under conditions similar to those
needed by the Lasso or Group-Lasso-type methods, but with a simplified
tuning strategy.  We implement our method via a new algorithm, with
proved convergence properties, which, unlike existing methods, scales
well with the dimension of the problem.  Our simulation studies
support strongly our theoretical findings.
\end{abstract}

\begin{IEEEkeywords}
Group Square-Root Lasso, high dimensional regression, noise level, sparse regression, Square-Root Lasso, tuning parameter
\end{IEEEkeywords}


\newcommand{\degree}{\ensuremath{^\circ}}
\newcommand{\pmes}{\mathcal{P}}
\newcommand{\mr}{\mathbb{R}}
\newcommand{\mrpn}{\mathbb{R}_0^+}
\newcommand{\mq}{\mathbb{Q}}
\newcommand{\mn}{\mathbb{N}}
\newcommand{\mz}{\mathbb{Z}}
\newcommand{\me}{\mathbb{E}}

\newcommand{\mca}{\mathcal{A}}
\newcommand{\mcb}{\mathcal{B}}

\newcommand{\mpr}{\mathbb{P}}
\newcommand{\mprn}{\mathbb{P}_n}

\newcommand{\si}{\sum_{i=1}^n}
\newcommand{\sino}{\frac{1}{n}\sum_{i=1}^n}
\newcommand{\sgr}{\sum_{j=1}^q\sqrt{T_j}}
\newcommand{\sgrs}{\sum_{j\in S}\sqrt{T_j}}
\newcommand{\sgrsc}{\sum_{j\in S^c}\sqrt{T_j}}
\newcommand{\See}{\Sigma_{1,1}}
\newcommand{\Sze}{\Sigma_{2,1}}
\newcommand{\Szet}{\widetilde\Sigma_{2,1}}
\newcommand{\Sez}{\Sigma_{1,2}}
\newcommand{\Szz}{\Sigma_{2,2}}
\newcommand{\Seem}{\Sigma_{1,1}^{\text -1}}
\newcommand{\Seemt}{\widetilde\Sigma_{1,1}^{\text -1}}
\newcommand{\ona}{\operatorname}
\newcommand{\sgn}{\operatorname{sgn}}
\newcommand{\bsb}{\boldsymbol}
\newcommand{\rd}{\,\mathrm{d}}

\newcommand{\undz}{\underline{Z}}
\newcommand{\ovez}{\overline{Z}}

\newcommand{\aln}[1]{\begin{align*}#1\end{align*}}
\newcommand{\eqn}[1]{\begin{equation*}#1\end{equation*}}
\newcommand{\al}[1]{\begin{align}#1\end{align}}
\newcommand{\eq}[1]{\begin{equation}#1\end{equation}}

\newtheorem{theorem}{Theorem}[section]
\newtheorem{proposition}{Proposition}[section]
\newtheorem{lemma}{Lemma}[section]
\newtheorem{corollary}{Corollary}[section]
\newtheorem{remark}{Remark}[section]
\newtheorem{example}{Example}[section]
\newtheorem{definition}{Definition}[section]

\newenvironment{pro}{\begin{proof}[Proof]}{\end{proof}}

\section{Introduction}

Variable selection in high dimensional  linear regression models has become a very active area of research in the last decade.  In
linear models one observes  independent response random variables  $Y_i \in \RR $, $1 \leq i \leq n$,  and assumes that each $Y_i$  can be written as a  linear function of  the $i$-th observation on a $p$-dimensional predictor vector  $X_i =: (X_{i1}, \ldots, X_{ij}, \ldots, X_{ip})$, corrupted by noise:
\begin{equation}\label{Model1}
Y_i = X_i\beta^0 +\sigma\epsilon_i,
\end{equation}
where $\beta^0\in \mr^p$ is the unknown
 regression vector, $\sigma\geq 0$ is
the noise level, and for each $1 \leq i \leq n$,  the additive term  $\epsilon _i$,  is a mean zero random noise component.  Postulating that some components of $\beta^0$ are zero is equivalent to assuming that the corresponding  predictors are unrelated to the response after controlling for the predictors with non-zero components.  The problem of predictor selection can be therefore solved by devising methods that estimate accurately where the zeros occur.

 More recently, a large literature focusing on the selection of groups of predictors has been developed. This problem requires methods that set to zero entire groups of coefficients and is the focus of this work. Group selection  arises naturally
whenever it is plausible to assume, based on scientific considerations, that entire subsets of the $X$-variables are unrelated to the response. More generally, the need for
setting groups of coefficients to zero is a building block in variable selection in general additive models and sparse kernel learning, as discussed in  Meier et al. \cite{Meier9} and Koltchinskii and Yuan \cite{KY10}, among others.  Another direct application is to  predictor selection in the multivariate response regression model
\begin{equation}\label{Model2}
Z =  UA + E,
\end{equation}
where $Z$ is an $n \times m$ matrix in which each row  contains measurements on an  $m$-dimensional random response vector,
$U$ is a  $n \times p$ observed matrix whose rows are the $n$  measurements of  a $p$-dimensional predictor, $E$ is the zero  mean noise matrix, and $A$ is the unknown coefficient matrix.  A predictor $U_j$ is not present in this model if the $j$-th row of $A$ is equal to zero.  Using the vectorization operator $vec$,  (\ref{Model2})  can be written  as $vec( Z') =  (U \otimes  I) vec(A') + vec(E')$.
Thus, if one treats rows of $A$ as groups, predictor selection in model (\ref{Model2}) can be regarded as group selection
in linear models of type (\ref{Model1}).

Perhaps the most popular method for group selection is the Group-Lasso, introduced by Yuan and Lin \cite{Yuan06} and further studied theoretically in a number of works, including Lounici et al. \cite{lounici11}, Wei and Huang \cite{Weihuang10}.
The method consists in minimizing the empirical square loss plus  a term proportional to  the sum of the Euclidean norms of groups of coefficients. Specifically,  let $Y = (Y_1, \ldots, Y_n)'$. We  denote by  $X\in\mr^{n\times p}$ the  matrix with rows $X_i$, $ 1 \leq i \leq n$, and refer to it in the sequel as the design matrix.  We  assign the individual columns of the design matrix and the
corresponding entries of the regression vector to groups. For this, we consider a
partition $\{G_1,\dots,G_q\}$ of $\{1,\dots,p\}$ into groups and denote the
cardinality of a group $G_j$ by $T_j$ and the minimal group size by
$T_{\min}:=\min_{1\leq j\leq q}T_j$. We then assign all
columns of the design matrix $X$ with indices in $G_j$ to the group $G_j$. The corresponding
matrix is denoted by $X^j\in\mr^{n\times T_j}$. Similarly, for any vector
$\beta\in\mr^p$, we assign all components of $\beta$
with indices in $G_j$ to the group $G_j$ and denote the corresponding
vector by $\beta^j\in
\mr^{T_j}$.  We define the
 active set as  \begin{equation}\label{S} S:=\{1\leq j\leq q:\beta^{0j}\neq 0\}.\end{equation}
 We will denote by $\|v\|_2$ the Euclidean norm  of a generic vector $v$. Let  $\lambda > 0$ be a given tuning sequence.
With this notation, the Group Lasso estimator is given by
\[
\bar\beta:=\argmin_{\beta\in\mr^p}\left\{\frac{\|Y-X\beta\|_2^2}{n}+\frac{\lambda}{n}\sgr \|\beta^j\|_2\right\}. \]

 Optimal estimation of $\beta^0$, $X\beta^0$ and $S$ via the Group-Lasso is very well understood, and we refer
 to B\"uhlmann and van de Geer \cite{Buhlmann11} for an overview. However, one outstanding problem remains, and it is connected  to the practical choice of $\lambda$ that leads, respectively, to optimal estimation with respect to each of these three aspects.  It is agreed upon that whereas choosing $\lambda$ via cross-validation will yield estimates with good prediction and estimation accuracy, this choice is not optimal for correct estimation of $S$.  A possibility  is to  determine first  the theoretical forms of the tuning parameter that  yield optimal performances, respectively, and then estimate the unknown quantities
 in these theoretical expressions. One important reason for which this approach has not become popular is the fact that
 the respective optimal values of $\lambda$ depend on $\sigma$, the noise level,  and the  accurate estimation of $\sigma$  when $p > n$  may be as difficult as the original problem of selection. A step forward has been made by  Belloni et al. \cite{Belloni11},  in the context of variable (not group) selection. They introduced the Square-Root Lasso (SRL) given below
 \[
\bar{{\bar\beta}}:=\argmin_{\beta\in\mr^p}\left\{\frac{\|Y-X\beta\|_2}{\sqrt n}+\frac{\lambda}{n} \sum_{l=1}^{p}|\beta_l|\right\}. \]
The consideration of  the square-root form  of the criterion was  first proposed by Owen \cite{Owen07} in the statistics literature, and a similar approach is the Scaled Lasso by Sun and Zhang \cite{SunZhang12}. Belloni et al. \cite{Belloni11} studied theoretically the estimation and prediction accuracy of the SRL estimator $\bar{\bar{\beta}}$, and showed that
it is similar to that of the Lasso,  with the net advantage that optimality can be achieved for a tuning sequence independent  of $\sigma$. This makes this version of the Lasso-type procedure much more appealing when $p$ is large, especially when $p > n$, and opens the question whether the same holds true for pattern recovery, which was not studied in  \cite{Belloni11}. Moreover, given the wide applicability of group selection methods, it motivates the study of a grouped version of the Square-Root Lasso. We therefore introduce and study the Group Square-Root Lasso (GSRL)
\eq{\label{eq.yoyo.estimator}
\hat\beta:=\argmin_{\beta\in\mr^p}\left\{\frac{\|Y-X\beta\|_2}{\sqrt n}+\frac{\lambda}{n}\sgr \|\beta^j\|_2\right\}.
}
Our contributions are: \\

\noindent (a) To extend the ideas behind the Square-Root Lasso for  group selection and develop a new method, the
 Group Square-Root Lasso (GSRL). \\
\noindent (b) To show that  the GSRL estimator has optimal estimation and prediction, achievable with a $\sigma$-free tuning sequence $\lambda$. This generalizes the results for SRL obtained by  \cite{Belloni11}. \\
\noindent (c) To show that GSRL leads to correct pattern recovery,  with a $\sigma$-free tuning sequence $\lambda$.
This provides, in particular, a positive answer to the question left open in   \cite{Belloni11}.\\
\noindent (d) To propose algorithms with guaranteed convergence properties that scale well with the size of the problem,  measured by $p$, thereby extending the scope of the existing procedures, which are performant mainly  for small and moderate values of $p$.

We address (a), (b)  and (c) in Section 2 below, and (d) in Section 3. Section 4 contains simulation results that support strongly our findings. The proofs of all our results are collected in the Appendix. 


\section{Theoretical Properties of the Group Square-Root Lasso}

In this section, we show that: (i) Nothing is lost by using $\|Y - X\beta\|_2$ instead of $\|Y - X\beta\|_2^2$ in the definition of   our estimator $\hat{\beta}$ given by (\ref{eq.yoyo.estimator}). Specifically,  the Group Square-Root Lasso has the same accuracy as the  Group Lasso, under essentially the same conditions, in terms of estimation, prediction and subset recovery.  (ii)  The net gain is that these properties are achieved via a tuning parameter
$\lambda$ that is $\sigma$-free, in contrast with the Group-Lasso, which requires a tuning parameter $\lambda$ that is a function of  $\sigma$.

The following notation and conventions will be used throughout the paper. We assume that
the design matrix is nonrandom and normalized such that the diagonal
entries of the Gram matrix
$\Sigma:=\frac{X'X}{n}$ are equal to 1.  We denote the cardinality of the set $S$ defined in (\ref{S}) above by $s$, that is $|S| = s$, and refer to $s$ as the sparsity index.  We set  $s^*:=\sum_{j\in S}T_j$. We
 denote by  $\beta_S\in\mr^s$ (and similarly $\beta_{S^c}\in\mr^{p-s}$) the vector that consists of the entries of
 $\beta\in\mr^p$ with indices in $\bigcup_{j\in S}G_j$ (or $\bigcup_{j\in S^c}G_j$). Corresponding notation
 is used  for matrices. For a generic vector $v$ we denote by  $\|v\|_{\infty}$ its supremum norm, the maximum absolute value of its coordinates. \\

\subsection{Estimation and Prediction}

We begin with the study of the estimation and prediction accuracy of the Group Square-Root Lasso.  We first state and discuss the conditions under which these results will be established.

As shown in Theorem \ref{thm.yoyo.pred}  below,  our results
hold under  the general  Compatibility Condition on the design matrix, introduced for the Lasso in \cite{vandeG07},  and extended to this setting in \cite[Page~255]{Buhlmann11}. This condition is a slight relaxation of the
widely used Cone or Restricted Eigenvalues Condition (see \cite{Bickel09}).  We refer to  \cite{Sara09} and  \cite[Chapter 6.13]{Buhlmann11} for a  detailed comparison between these two and other related conditions. 

\paragraph{Compatibility Condition (CC)} We say that the Compatibility
Condition is met for $\kappa>0$ and $\gamma>1$ if
\eq{\label{eq.yoyo.compconst}
\sgrs\|\delta^j\|_2\leq \frac{\sqrt {s^*}\|X\delta\|_2}{\sqrt n\kappa}
}
for all $\delta\in\Delta_{\gamma}$, where
\begin{equation}
  \Delta_{\gamma}:=\{\delta\in\mr^p:\sgrsc\|\delta^j\|_2\leq \gamma \sgrs \|\delta^j\|_2\}.
\end{equation}

\noindent We  refer to $\kappa$ and $\gamma$ as the compatibility constants and
write \[(\kappa,\gamma)\in C(X,S). \]The compatibility constant $\kappa$ measures the correlations in the
design matrix: the smaller the value of  $\kappa$,  the larger the correlations.


For clarity of exposition, we will assume for the rest of the paper that the additive noise terms $\epsilon_i$ have a standard Gaussian distribution.

The second ingredient in our analysis is  the  definition of the appropriate noise component  that needs to be compensated for by the tuning parameter $\lambda$. The proofs of our results reveal that it  is
\begin{equation}
V:=\max_{1\leq j \leq q}\left\{\frac{\sqrt n\|(X'\epsilon)^j\|_2}{\sqrt{T_j}\|\epsilon\|_2}\right\}.
\end{equation}\label{v}
For $\gamma > 1$ given by condition {\bf CC}  above , let $\overline\gamma
:=\frac{\gamma+1}{\gamma-1}$. For given $\lambda > 0$  define the
set
\begin{equation}\label{a}\mathcal{A}: = \left\{ V \leq \lambda/\overline\gamma\right\}.\end{equation}

We first establish our result over the set $\mathcal{A}$. We then  show, in Lemma \ref{lambdagauss} below, that the set $\mathcal{A}$  has probability $1-\alpha$, for any $\alpha$ close to zero, for an appropriate choice of the tuning parameter $\lambda$. Since $\lambda$ will be chosen relative to the ratio of the random variables that define $V$,   the factor  $\sigma$ cancels  out. This  is the
key for obtaining a tuning parameter $\lambda$ independent of the variance of the noise. \\

 \noindent
 With $\gamma > 1$ given by {\bf CC} above and $\kappa >0$ given by {\bf CC},
we assume in what follows that the sparsity index $s^*$ is not larger than the sample size $n$. Specifically, we assume that
\begin{equation}\label{less}  s^* < \frac{n^2\kappa^2}{\lambda^2}. \end{equation}
We will show in Lemma \ref{lambdagauss}  below that the value of $\lambda$ for which the event  $\mathcal{A}$ has high probability is, in terms of orders of magnitude,  no larger than  $\lambda = O(\sqrt{n\log q})$.  Therefore, and using the notation $\lesssim$ for inequalities that hold up to multiplicative constants, the condition on the sparsity  index becomes
\[ s^* \lesssim \frac{n}{\log q }, \]
which re-emphasizes the introduction of $s^*$ in this analysis to start with: whereas we allow $ p > n$, we cannot expect
good performance of any method from a  limited sample size $n$, unless the true model has  essentially fewer parameters than $n$.

The following result summarizes  the prediction and estimation properties of the Group Square-Root Lasso estimator.  It generalizes \cite[Theorem~1]{Belloni11}, where the Square-Root Lasso is treated, corresponding in our set-up to the special case  $q=p$.

\begin{theorem}\label{thm.yoyo.pred} Assume that
  $(\kappa,\gamma)\in C(X,S)$ and that (\ref{less}) holds.  Then, on the event $\mathcal{A}$,  the following hold: 
\eqn{
\|X(\widehat\beta-\beta^0)\|_2 \lesssim \frac{\sigma \lambda\sqrt{s^*}}{\kappa\sqrt{n}}
}
and
\eq{\label{eq:yoyoest}
  \sgr\|(\widehat\beta-\beta^0)^j\|_2  \lesssim \frac{\sigma\lambda s^*}{\kappa^2n}.
}
\end{theorem}
\noindent The precise constants in the statements above are given in the proof of this theorem, presented in the appendix.
Theorem \ref{thm.yoyo.pred}  is the crucial step in showing  that the GSRL estimator, which has a tuning parameter free of $\sigma$,  has the same optimal rates of convergence as the Group Lasso estimator, see for instance Lounici et al. \cite{lounici11} or B\"uhlmann and van de Geer \cite{Buhlmann11}. We will  determine the size of $\lambda$ in Lemma \ref{lambdagauss}  below and state the resulting rates in Corollary \ref{sum}.\\

\begin{remark}
For prediction, the condition
  \begin{equation*}
    \sgrs\|\delta^j\|_2-\sgrsc\|\delta^j\|_2\leq \frac{\sqrt {s^*}\|X\delta\|_2}{\sqrt n\kappa}
  \end{equation*}
for $\delta\in\Delta_1$ could replace the CC condition \eqref{eq.yoyo.compconst}, cf.~\cite{Bellarx}. We additionally note that prediction (in contrast to correct subset recovery and estimation) is even possible for highly correlated design matrices, see \cite{MomoYoyo12, vdGeer11}. However, a detailed discussion of prediction for correlated design matrices is not in the scope of this paper.
\end{remark}
\begin{remark}
Inequality~\eqref{eq:yoyoest} directly implies correct subset recovery for the Group Square-Root Lasso in the special case $\sigma=0$, cf.~\cite{Bellarx}. In contrast, $\sigma=0$ and the conditions of Theorem~\ref{thm.yoyo.pred} are not sufficient to ensure correct subset recovery for the Lasso and the Group Lasso.
\end{remark}



\subsection{Correct subset recovery}

We study below the subset recovery properties of the Group Square-Root
Lasso.  Similarly to the analysis of all  other Lasso-type procedures,  subset recovery can only be guaranteed under
additional assumptions on the model.

The first condition is the  Group
Irrepresentable Condition, which is an additional condition on the
the design matrix $X$. To introduce it, we decompose the
 Gram matrix $\Sigma$ with $\Sigma_{1,1}:=\frac{X_S'X_S}{n}$,
$\Sigma_{1,2}:=\frac{X_S'X_{S^c}}{n}$,
$\Sigma_{2,1}:=\frac{X_{S^c}'X_{S}}{n}$, and
$\Sigma_{2,2}:=\frac{X_{S^c}'X_{S^c}}{n}$. We  define $\Szet:=(0~~\Sez)'$
and $\Seemt:=(0~~\Seem)'$.

\paragraph{Group Irrepresentable Condition (GIR)} We say that the Group
Irrepresentable Condition is met for $0<\eta<1$  if $\Sigma_{1,1}$ is
invertible and
\eq{\label{eq.yoyo.groupirr}
\max_{v:\|v^k\|_2\leq \sqrt{T_k}}\max_{1\leq j \leq q}\frac{\|(\Szet\Seem v)^j\|_2}{\sqrt{T_j}}<\eta.
}
We refer to $\eta$ as the group irrepresentable constant and
write \[ \eta\in I(X, S).\]

\noindent

The Group Irrepresentable Condition implies  the  Compatibility
Condition discussed above, see for instance  \cite{Buhlmann11},  and it is therefore more restrictive. However,  it is  essentially a
necessary and sufficient condition  for consistent support recovery via  Lasso-type procedures, see
\cite{BiYuConsistLasso}. We refer to   \cite{Buhlmann11, BiYuConsistLasso,
Meinshausen06, Zou06} for  different versions of the Irrepresentable
Condition and further discussion of  these versions.

The second condition needed for precise support recovery regards the strength of the signal $\beta^0$. Because the noise can conceal small components of the  regression
vector $\beta^0$, some of its  nonzero components need to be
sufficiently large to be detectable. We formulate
this in the Beta Min Condition, similarly to  \cite{BuneaEN} and \cite{Schelldorfer11,
  Sara11}:
\paragraph{Beta Min Condition (BM)} We say that the Beta Min Condition is met for $m\in\mr^s$  if
\eq{\label{eq.yoyo.betamin}
\|\beta^{0j}\|_\infty\geq m_j,
}
for all $j \in S$.
We then write $m\in B(\beta^0)$.\\

\noindent   Note that only one component of $\beta^0$ in each non-zero group has to be sufficiently large,  because we aim to select whole groups,  and not individual components.\\

\noindent  A slightly different tuning parameter, still independent of $\sigma$ is needed for consistent subset recovery. Let  $\widetilde\eta:=\frac{1+\eta}{1-\eta}$, for $\eta$ given by GIR above, and recall that
$\overline{\gamma} = \frac{\gamma+1}{\gamma-1}$, with $\gamma$ defined in condition CC above.  Define the event

\begin{equation}\label{a1}
\mathcal{A}_1 = \{ V\leq \lambda/(\overline\gamma\vee 2\widetilde\eta)\}. \end{equation}

\noindent Finally, we introduce the following notation
\eqn{
\xi_{\|\cdot\|_{\infty}}:=\max_{v:\|v^k\|_2\leq \sqrt{T_k}}\max_{1\leq j\leq q} \frac{\|(\Seemt v)^j\|_{\infty}}{\sqrt{T_j}}.
}
Note that for orthonormal design matrices,
$\xi_{\|\cdot\|_\infty}=1$. 
Let $\alpha \in (0, 1)$ be given. \\

\begin{theorem}\label{thm.yoyo.varsel}
Assume that  the conditions CC, GIR and BM are met, and that (\ref{less}) holds. Assume that $(\kappa,\gamma)\in C(X,S)$ and $\eta\in I(X,S)$.  Let $D > 0$ be a dominating constant. Then, on the set $\mathcal{A}_1$, we have, with probability greater than $1 - \alpha$: \\

\noindent (1) \  $\widehat\beta_{S^c}=0$; \\
\noindent (2) \
For all $1\leq j \leq q$,
\eqn{
  \|(\widehat\beta-\beta^0)^j\|_{\infty}\leq
  D \frac{\sqrt{T_j}\sigma\lambda}{n}.
}
\noindent (3) \ If there exists an $m\in B(\beta^0)$ such that $m_j \geq  D\frac{\sqrt{T_j}\sigma\lambda}{n }$, for each  $ j \in S$, then
\eqn{
S= \widehat S.
}
\end{theorem}
\begin{remark}\textnormal{The constant $D$  depends on $\gamma, \eta, \kappa$ and $\xi_{\|\cdot\|_\infty}$, but not on $n, p, q$.  Its exact form is given in the proof of  Theorem \ref{thm.yoyo.varsel}. The results above show that the Group Square Root Lasso will recover the sparsity pattern consistently, as long as $\mathcal{A}_1$ has high probability, which we show in Lemma \ref{lambdagauss} below. Theorem   \ref{thm.yoyo.varsel}  holds under slightly more general  conditions on the design  than the variant on  the mutual coherence condition employed in Lounici et al. \cite{lounici11}, for pattern recovery with the Group Lasso.  Moreover, the recovery is guaranteed for signals of minimal strength,  just above noise level,  which we quantify precisely  in Corollary \ref{sum}  below.  }
\end{remark}

\begin{remark}\textnormal{
Theorem \ref{thm.yoyo.varsel} can be proved only under  GIR and BM, as GIR  implies CC.  However, using only GIR would require the derivation of the corresponding constants under which CC holds, as we will appeal to the conclusion
of Theorem \ref{thm.yoyo.pred}  in the course of the proof of Theorem  \ref{thm.yoyo.varsel}. Given that the arguments are already technical, we opted for stating both assumptions separately, for transparency. }
\end{remark}

\begin{remark}\label{lemma.yoyo.vs3} \textnormal{The Group Square-Root Lasso can be shown to lead to  correct subset recovery under sharper  Beta Min Conditions, for a constant $D$  independent of  $\xi_{\|\cdot\|_\infty}$, if we impose stricter conditions on the design. For example, one can invoke the  Group Mutual Coherence Condition (GMC) and apply ideas developed in \cite{BuneaEN} to find the condition $m_j\gtrsim\sqrt{T_j}\lambda /n$, which is of the same order as above,  but holds up to universal constants, independent of the conditions on the design. We do not detail this approach here, since the GMC  implies GIR,  and the proof would follow very closely the ideas in \cite{BuneaEN}.}
\end{remark}

\subsection{Choice of the Tuning Parameter}
\label{sec.tuningparameter}
As discussed above, the novel  property of the Group Square-Root Lasso  method is that its tuning parameter $\lambda$ can be chosen independently of the noise level $\sigma$. This is particularly interesting in the high-dimensional setting $p\gg n$, where good estimates of $\sigma$ are not usually available.  In determining  $\lambda$ for this method, we recall that it has to be  sufficiently large to overrule the noise component, which is independent of $\sigma$, \eqn{
V=\max_{1\leq j \leq q}\left\{\frac{\sqrt n\|(X'\epsilon)^j\|_2}{\sqrt{T_j}\|\epsilon\|_2}\right\},
}
in that the events $\mathcal{A}$ and $\mathcal{A}_1$, given above by (\ref{a}) and (\ref{a1}), respectively, hold with high probability.  At the same time,  the bounds in Theorem \ref{thm.yoyo.varsel} and \ref{thm.yoyo.pred} become sharper
for smaller values of $\lambda$. To incorporate these two constraints, we choose
the tuning parameter as the smallest value that overrules the noise part
with high probability.  For this, we fix $\alpha\in(0,1)$ and choose the
smallest value $\lambda$ such that with probability at least $1-\alpha$ it
still holds that $\lambda/\overline\gamma\geq V$ or
$\lambda/(\overline\gamma\vee 2\widetilde\eta)\geq V$,  depending on the
type of results we are interested in. Standard values for $\alpha$ are
$0.05$ and $0.01$. \\

\noindent For each $j$, let  $\zeta_j= \|X^j\|^2 /n$ and $\zeta =\max_{j} \zeta_j$, where $\| A \|$ is the operator norm of a generic matrix $A$.

\begin{lemma}\label{lambdagauss}  Assume that the noise terms $\epsilon_i$,  $1 \leq i \leq n$,  are i.i.d.   standard  Gaussian random variables,  and assume that $T_j < n$, for all $ 1 \leq j \leq q$.  Let $\alpha \in (0, 1)$ be given such that $16\log(2q/\alpha)\leq n - T_{\text{max}}$. Then, if \[ \lambda_0 \geq \frac{\sqrt{2\zeta}n}{\sqrt{n-T_\text{max}}}\left(1+\sqrt{\frac{2\log(2q/\alpha)}{T_\text{min}}}\right),\]
 it holds that
\eqn{
\mpr(V\geq \lambda_0)\leq \alpha.
}
\end{lemma}

\noindent  As an immediate consequence, the following corollary summarizes the expressions of $\lambda$ for which the
events $\mathcal{A}$ and $\mathcal{A}_1$ hold with probability $1 - \alpha$, for each given $\alpha$. Notice that $\lambda$ is independent of $\sigma$, as claimed. Corollary \ref{sum} also shows that the  sharp rates of convergence and subset recovery properties of  the Group Lasso are also enjoyed by the Group Square-Root Lasso, with the important added benefit that  the new method's tuning parameter is $\sigma$-free.

\begin{corollary} \label{sum}
 Assume that the noise terms $\epsilon_i$, $1 \leq i \leq n$,  are i.i.d.   standard Gaussian random variables and assume that $T_j < n$, for all $ 1 \leq j \leq q$.  Let $\alpha \in (0, 1)$ be given such that $16\log(2q/\alpha)\leq n - T_{\text{max}}$.\\

\noindent  (i) If
  \begin{equation*}
 \lambda \geq \frac{\sqrt{2\zeta}n\overline\gamma}{\sqrt{n-T_\text{max}}}\left(1+\sqrt{\frac{2\log(2q/\alpha)}{T_\text{min}}}\right),
  \end{equation*}
then $\mpr\left(\mathcal {A} \right)\geq 1-\alpha$. \\
\noindent (ii) If
  \begin{equation*}
 \lambda\geq \frac{\sqrt{2\zeta}n(\overline\gamma\vee 2\widetilde\eta)}{\sqrt{n-T_\text{max}}}\left(1+\sqrt{\frac{2\log(2q/\alpha)}{T_\text{min}}}\right),      
\end{equation*}
then  $\mpr\left(\mathcal A_1 \right)\geq 1-\alpha$.\\
\noindent (iii) Under the assumptions of  Theorem ~\ref{thm.yoyo.pred}, its conclusion holds  with probability at least $1 - 2\alpha$ and $ \lambda = O(\sqrt{\frac{ n}{T_\text{min}}} \log q ).$ \\
\noindent (iv) Under the assumptions of  Theorem ~\ref{thm.yoyo.varsel}, its conclusion holds  with probability at least $1 - 2\alpha$ and $ \lambda = O(\sqrt{\frac{n}{T_\text{min}}}\log q ).$ \\
\end{corollary}
\noindent The first two claims follow immediately from Lemma \ref{lambdagauss} and the definitions of $\mathcal{A}$ and $\mathcal{A}_1$, respectively. The  third and forth claims follow directly from the first two,  by invoking Theorems ~\ref{thm.yoyo.pred} and  ~\ref{thm.yoyo.varsel}, respectively.
We only considered Gaussian noise above for clarity of exposition. However, more general results can be established applying different deviation inequalities, for instance \cite{Bousquet02, YoyoSara12, vdGeer11b}.  For example, if the $\epsilon_i$'s belong to a general sub-exponential family, the order of magnitude of $\lambda$ remains the same. We also refer to~\cite{Bellarx}, where the analysis involving non-Gaussian noise  makes use of  moderate deviation theory for self-normalized sums,  leading in some cases to results similar to those obtained for Gaussian noise. Additionally,  an analysis that takes into account correlations between the groups is  expected to lead to  results similar to  those established for the Lasso, see \cite{MomoYoyo12, vdGeer11}.

\section{Computational Algorithm}
\label{sec:comp}
%

In this section we show that  the Group Square-Root Lasso can be implemented very efficiently.  We consider estimators of a  form slightly more general than (\ref{eq.yoyo.estimator}):
\begin{align}
\hat \beta := \argmin_{\beta\in\mr^p}\left\{ \|Y - X \beta\|_2 + \sum_{j=1}^q \lambda_j \|\beta^j\|_2\right\}, \label{sqrt-lasso}
\end{align}
where $\lambda_1,\dots,\lambda_q>0$ are arbitrary given constants.
For convenience, we will implement, without loss of generality, the following variant
\begin{align}
\hat \beta  := \argmin_{\beta\in\mr^p}\left\{ \|Y - X \beta\|_2/K + \sum_{j=1}^q \lambda_j \|\beta^j\|_2\right\}, \label{sqrt-lasso2}
\end{align}
where $K$  is a   fixed, sufficiently large constant.
A global minimum of \eqref{sqrt-lasso2}, for given constants
$\lambda_1,\dots,\lambda_q$,  is also a global minimum of \eqref{sqrt-lasso}
with constants $K\lambda_1,\dots,K\lambda_q$.\\ 

When $q=p$ and $\lambda_1=\dots=\lambda_q=\lambda$, \eqref{sqrt-lasso2} reduces to the Square-Root Lasso, which was
formulated  in the  form \cite{Belloni11}:
\begin{align}
\min_{t, v, \beta^+, \beta^-} &\frac{t}{K} + \lambda \sum_{j=1}^p (\beta^{j+} + \beta^{j-}) \notag\\\quad & s.t. \quad
v_i = Y_i - x_i^T \beta^+ + x_i^T \beta^-,~ 1\leq i \leq n, ~t \geq \|v\|_2,~ \beta^+\geq 0,~ \beta^-\geq 0.
\end{align}
The last three constraints are second-order cone constraints.
Based on this conic formulation, Belloni et al.  \cite{Belloni11},  have derived three computational algorithms for solving the Square-Root Lasso:
\begin{enumerate}
\item First order methods by calling the TFOCS Matlab package, or TFOCS for short;\item Interior point method by calling the SDPT3 Matlab package, or IPM for short;
\item Coordinatewise optimization, or COORD for short.
\end{enumerate}
According to our experience, TFOCS is very slow and inaccurate.  COORD is reasonably
fast,  but not as accurate as IPM, especially in applications with a large
number of parameters $p$. In computing a solution path,  COORD is still much slower  than the,  perhaps most popular,  coordinate descent algorithm  for solving the Lasso  \cite{FHHT07}.
Therefore, even for the Square-Root Lasso, without groups, a fast and accurate algorithm  is still needed.\\

We propose a  scaled thresholding-based iterative selection procedure (\textbf{S-TISP}) for solving the general Group Square-Root Lasso problem \eqref{sqrt-lasso2}.
Assume  the scaling step
\begin{align}
Y \leftarrow Y / K, \quad X \leftarrow X/K \label{scaling}
\end{align}
has been performed. Starting from an arbitrary $\beta{(0)}\in \mathbb R^p$,
S-TISP performs  the following iterations to update $\beta(t)$, $t=0, 1,
\dots$:
\begin{align}
\beta^j (t+1) = \vec\Theta( \beta^j {(t)} + (X^j)' (Y -    X \beta {(t)}) ; \lambda_j \| X \beta {(t)} - Y\|_2 ), \quad 1\leq j \leq q. \label{tisp-iter}
\end{align}
Here, $\vec\Theta$ is the multivariate soft-thresholding operator
\cite{She12} defined through $\vec\Theta(0;\lambda):=0$ and $\vec\Theta(a;\lambda):=
a \Theta(\|a\|_2;\lambda)/\|a\|_2$ for $a\neq  0$, where
$\Theta(t;\lambda):=\mbox{sign}(t) (|t|-\lambda)_+$ is the
soft-thresholding rule. S-TISP is extremely  simple to implement and does
not resort to any optimization packages.\\

The following theorem guarantees   the global convergence of $\beta{(t)}$.
The result is   considerably stronger  than  those `every accumulation point'-type
conclusions that are often seen in numerical analysis. 
\begin{theorem} \label{th:stisp-conv}
Suppose $\lambda_j>0$ and the following regularity condition holds:
$\inf_{\xi \in A} \|X \xi - Y \|_2 > 0$, where $A= \{\vartheta \beta{(t)}+(1-\vartheta)\beta{(t+1)}:\vartheta  \in [0, 1], t=0, 1, \dots\}$.
Then, for $K$ large enough, the sequence of iterates $\beta{(t)}$ generated
by \eqref{tisp-iter} starting with any $\beta^{(0)}$ converges to a global minimum  of \eqref{sqrt-lasso2}.
\end{theorem}

According to our experience, smaller values of $K$
lead to faster convergence if the algorithm converges. The choice $K = \| X \|/\sqrt 2$,
motivated by display \eqref{errbnd} in the proof of Theorem \ref{th:stisp-conv}, works well  in the
simulation studies; we recall that $\|X\|$ is the operator norm of the matrix $X$.  The associated objective function
 is  $ \| Y - X
\beta\|_2 + \sum_j  K \cdot \lambda_j \| \beta^j \|_2$ which reduces to the specific form  (\ref{eq.yoyo.estimator}) if  we set \[\lambda_j =  \frac{\lambda}{K} \cdot \sqrt{\frac{T_j}{n}}.\]
Other choices of $\lambda_j$ are allowable  in our computational algorithm.
 We
suggest using \emph{warm starts} so that the convexity of the problem can
be well exploited. Concretely, after specifying a decreasing grid for $\lambda$, denoted by $\Lambda = \{\lambda_1, \cdots, \lambda_l\}$, we use the converged solution $\hat \beta_{\lambda_l}$  as the initial point $\beta{(0)}$ in \eqref{tisp-iter} for the new optimization problem associated with $\lambda_{l+1}$.


\section{Simulations}

\subsection{Computational Time Comparison for the Square-Root Lasso}

As explained  in Section \ref{sec:comp} above, the Square-Root Lasso can be
computed using one of the algorithms TFOCS, COORD, or IPM  \cite{Belloni11}. As the Square-Root
Lasso is a special case of the Group Square-Root Lasso (GSRL), corresponding to  $q=p$,  it can also be implemented via our
proposed S-TISP algorithm.  In this section we compare the three existing methods with ours in terms of computational time.
 We are particularly interested in high-dimensional, sparse problems, when  $p$ is
large and $\beta^0$ is sparse.  Since no competing GSRL algorithms exist,  we only consider the non-grouped version of  S-TISP in the experiments below, for transparent comparison with published literature on algorithms for  the Square-Root  Lasso, which  is only devoted to variable selection, and not to group selection. \\


For uniformity of comparison, we used a Toeplitz design as in Belloni et al.  \cite{Belloni11} with
correlation matrix $[0.5^{|i-j|}]_{p\times p}$. The noise variance is fixed
at $1$ and the true signal is  the $p$-dimensional, sparse vector  $\beta^0 = \begin{array}{ccccccc} ( 2.5 & 0 & 2.5 & 2.5 & 0 & \cdots & 0 )' \end{array}$.  The first four components of $\beta^0$ are fixed. The rest are all equal to zero, and their number
varies   as we vary the dimension of $\beta^0$ by setting  $p=25,50,100,200,500,1000$, in order  to investigate the computational scalability of each of the algorithms under consideration.  We set $n=50$ for all values of $p$.
We perform the following computations: \\
(i) \emph{PATH.} Solution paths are computed for $\lambda /(\sqrt n K)=2^{-6}, 2^{-5.8}, \cdots, 2^{-0.2}, 2^0$. This grid is empirically constructed  to cover  all potentially interesting solutions as  $p$ varies.
\\
(ii) \emph{TH.} We use the  theoretical choice $\lambda/\sqrt n=1.1 \Phi^{-1}(1-0.05/(2p))$ recommended in  \cite{Belloni11}   to compute a specific coefficient estimate. In both cases, the error tolerance is  1e-6.  Each experiment is repeated 50 times, and we report the average CPU time.

We used the Matlab codes downloaded from Belloni's website and installed
some further required Matlab packages,  with necessary changes to rescale
$\lambda$. We  made consistent termination criteria, and suppressed the
outputs. In particular, we implemented the warm start initiation in COORD
which boosts its convergence substantially. The original initialization in
the COORD relies on a ridge regression estimate and is slow in computing a
solution path. Table \ref{tab1} shows the average computational time for 50
runs of each of the algorithms  under comparison.

\begin{table}[ht!]
\centering
\setlength{\tabcolsep}{1mm}
\caption{\small{Computational time comparison (CPU time in seconds) of the first order method by calling the TFOCS  package, the interior point method (IPM) based on SDPT3, the coordinatewise optimization (COORD), and the S-TISP.} }\label{tab1}

\small{
\begin{tabular}{r | rrrrrr}
\hline\hline
\emph{PATH} & $p=25$ & $p=50$ & $p=100$ & $p=200$ & $p=500$ & $p=1000$ \\
\hline
S-TISP & 0.09 &   0.34  &  0.64 & 0.77 &  1.28 &  3.42\\
COORD & 0.24  &  0.67 &   0.68  &  0.69  &  2.37 &  32.13\\
IPM &  3.84 & 4.42 &  4.99 &  6.09 &   9.07 &  15.36 \\
TFOCS & 119.08 & 245.74 & 452.82 & 685.25 & 749.55    &   696.45\\
\hline\hline
\emph{TH} & $p=25$ & $p=50$ & $p=100$ & $p=200$ & $p=500$ & $p=1000$ \\
\hline
S-TISP &     0.02 &   0.04 &   0.13 &   0.30 &   0.66 &   1.80\\
COORD &    0.03  &  0.07 &   0.14   & 0.31  &  1.03   & 2.75\\
IPM &    0.13   & 0.15  &  0.20    &0.28   & 0.67   & 2.16\\
TFOCS &    1.00  &  1.45 &  1.42    &4.52 &   3.23 &   5.74\\
\hline
\end{tabular}
}
\end{table}

As we can see from Table \ref{tab1}, TFOCS and IPM 
do not scale well for growing $p$, especially when $p>n$. 
After comparing the COORD estimates to those obtained by
interior point methods (SDPT3 and SeDuMi), we found that, unfortunately, COORD is a  very
crude  and inaccurate approach. Its inaccuracy is exacerbated by  warm
starts.   We also found that the solutions obtained by calling the TFOCS
package are not trustworthy for moderate or large values of $p$,  and  that TFOCS
is very slow. Our S-TISP achieves comparable accuracy to IPM in the above
experiments, and its computational costs scale well with the problem size.
In fact, it provides an impressive computational gain  over the
aforementioned algorithms for high-dimensional data, that is, large $p$. 


\subsection{Tuning Comparison}
In this part of the experiments, we provide empirical evidence of  the advantages
of the Group Square-Root Lasso in  parameter tuning.\\

We use the same Toeplitz design as before and set $\sigma=1$.  The true
coefficient vector is generated as
$\beta^0=\left(\{2.5\}^3,\right.$ $\left. \{0\}^3, \{2.5\}^3,
    \{2.5\}^3, \{0\}^3, \cdots, \{0\}^3\right)'$ consisting of
three $2.5$'s, three $0$'s, three $2.5$'s, three $2.5$'s, and finally a
sequence of three $0$'s.  Hence,  $S=3$ and the group sizes are equal to 3. We
fix  $n=100$ and vary $p$ at $60, 300, 600$.
Each setup is simulated 50 times, and at each run, the Group Square-Root Lasso algorithm,  implemented through our proposed S-TISP, is called with three parameter tuning strategies. \\

\noindent (a) Theoretical choice, denoted by TH. This is based on a simplified version
of  the sequence $\lambda_0$ given by Lemma \ref{lambdagauss}. To motivate our choice, we first recall the notation  $\zeta_j= \|X^j\|^2 /n$ and $\zeta =\max \zeta_j$, where $\| A \|$ is the spectral norm of a generic matrix $A$. Define $T_{\min} =\min T_j$, $T_{\max} = \max T_j$. With this notation,  we showed in the course of the proof of  Lemma \ref{lambdagauss} that  the sequence   $\lambda_0$ needs to be chosen such that,
for given $\alpha$,
\begin{eqnarray*}
\mpr\left(V\geq \lambda_0\right)
& \leq & \ \sum_{j=1}^{q} \mpr\left(\chi^{2}_{T_{j}}   \geq  \frac{ \frac{\lambda_0^2 }{n^2} (n - T_{\max})}
{ (\zeta - \frac{\lambda_0^2 T_{\min}}{n^2})_{+} } \cdot \chi^{2}_{n - T_j} \right) \leq \alpha,
\end{eqnarray*}
where $\chi^{2}_{T_{j}}$ and  $ \chi^{2}_{n - T_j}$ are independent $\chi^2$ variables. Since the ratio of these two variables has a $F$-distribution, and with the notation $\tau := \frac{ \frac{\lambda_0^2 }{n^2} (n - T_{\max})}{ (\zeta - \frac{\lambda_0^2 T_{\min}}{n^2})_+}$, we further have
\begin{eqnarray*}
\mpr\left(V\geq \lambda_0\right)
& \leq & \ \sum_{j=1}^{q} (1 - F_{T_j, n-T_j} \left(  \tau \right) )\\
& \leq & \ q \left( 1 - F_{T_{\min}, n-T_{\min}} \left( \tau\right)\right),
\end{eqnarray*}
where $F_{n_1, n_2}$ denotes   the cumulative distribution function of a  F-distribution with $n_1$ and $n_2$ degrees of freedom. Hence, $\mpr\left(V\geq \lambda_0\right) \leq \alpha$ if  $\tau \geq F_{T_{\min}, n-T_{\min}}^{-1} (1-\alpha/q)=:\tau_0$ or, equivalently, if
\begin{equation}\label{true} \lambda_0\geq n \sqrt{ \frac{ \zeta\tau_0 }{T_{\min} \tau_0 + n - T_{\max}}}.\end{equation}
  The proof of  Lemma  \ref{lambdagauss},
in which control of the event  $\left(V\geq \lambda_0\right)$ and the determination of $\lambda_0$ is done via deviation inequalities for $\chi^2$ random variables,
can be used to show that $\lambda_0$ given by (\ref{true}) above has the correct order of magnitude.  Since the calculation involving  the  F-distribution leading to (\ref{true})  is more precise, we advocate this choice for practical use, for models with Gaussian errors. We further
use Corollary \ref{sum} to choose $\lambda = \lambda_0$, for our particular design. \\

 Therefore, we use the form  \eqref{sqrt-lasso2} in our implementation, with  \[  \lambda_j =    \sqrt{\zeta \tau_0/(T_{\min} \tau_0+n -T_{\max})} \sqrt{ n T_j}/K, \]  and  $\tau_0 = F_{T_{\min}, n-T_{\min}}^{-1}(1-\alpha/q)$, $K = \|X\|_2/\sqrt 2$ and   $\alpha=0.01$.
 After the optimal estimate is located, bias correction is conducted  by fitting a local OLS restricted to the selected dimensions,  to boost  the prediction accuracy. \\

 \noindent  (b) Cross-Validation (CV). We use 5-fold CV to determine the optimal value of  $\lambda$ and the associated  estimate. Similarly, bias-correction is performed at the end.  \\

 \noindent (c) SCV-BIC \cite{She12}. We cross-validate the sparsity patterns instead of  the values of $\lambda$. Unlike $K$-fold CV, only one  penalized solution path  needs to be generated by running the  Group Square-Root Lasso on the \textit{entire} dataset. This  determines the candidate sparsity patterns. Then,  we fit restricted OLS    in each CV training to evaluate  the validation error of the associated sparsity pattern and append a BIC correction term to the total validation error. SCV-BIC is much less expensive than CV, noting that the  OLS fitting is cheap,  and has been shown to bring significant performance improvement, see \cite{She12} for details and \cite{bb} for a similar approach.\\

To measure  the prediction accuracy, we generated  additional test data with $N_{test}=$1e+4 observations in each simulation.
The effective prediction error is given by $\mbox{\textbf{MSE}}=100 \cdot (\sum_{i=1}^{N_{test}} (y_i - x_i^T \hat \beta)^2/(N_{test}\sigma^2)-1)$. We found the histogram of MSE is highly asymmetric and far from Gaussian. Therefore,  the $40\%$ trimmed-mean (instead of the mean or the somewhat crude median) of  $\mbox{MSE}$s was  reported  as the goodness of fit of the obtained model.
We characterize  the  selection consistency by computing the  \textbf{Miss} ($\mbox{M}$) rate -- the mean of $| \{ j: \beta^{j0} \neq 0,\hat{\beta}^{j0}=0\} | / | \{ j: \beta^{j0} \neq 0\} | $ over  all simulations, where $|\cdot|$ is the cardinality of a set,
and  \textbf{False Alarms} ($\mbox{FA}$) rate  -- the mean of $| \{ j: \beta^{j0} = 0,\hat{\beta}^{j0}\neq0\} | / | \{ j: \beta^{j0} = 0\} | $ over  all simulations. Correct selection occurs when M = FA = 0.

\begin{table}[ht!]
\centering
\setlength{\tabcolsep}{1mm}
\caption{\small{Performance of Group Square-Root Lasso Tunings---CV, SCV-BIC, and the theoretical choice (TH), in terms of miss rate (M), false alarm rate (FA), and prediction error (MSE).  } }\label{tab2}

\small{
\begin{tabular}{r | ccc | ccc | ccc}
\hline\hline
& \multicolumn{3}{c|}{$p=60$} & \multicolumn{3}{c|}{$p=300$} & \multicolumn{3}{c}{$p=600$}  \\
\hline
& M & FA & MSE & M & FA & MSE & M & FA & MSE \\
\hline
\textbf{CV} & 0\% & 12.75\% & 23.21 & 0\% & 2.56\% & 22.16 & 0\% & 0.80\% & 18.36\\
\textbf{SCV-BIC} &0\% & 0\% & 9.82 & 0\% & 0.02\% & 9.99 & 0\% & 0.01\% & 8.95\\
\textbf{TH} & 0\% & 0\% & 9.82 & 0\% & 0\% & 9.99 & 0.67\% & 0\% & 9.20\\
\hline
\end{tabular}
}
\end{table}

The missing rates are very low,  which indicates that all truly relevant predictors are detected most of the time.  We point out that this  will typically happen when  the signal strength is moderate to high (2.5 in our simulations), and it supports our theoretical findings.  We expect a lesser performance when the signal strength is weaker. 
We conclude from Table \ref{tab2} that the selection by CV is acceptable,
especially in high-dimensional, sparse problems, but it has the worst behavior relative to the other tuning strategies.
 SCV-BIC gives excellent prediction accuracy and
recovers the true sparsity pattern successfully. It is much more efficient than CV but still requires the computation of one Group Square-Root Lasso solution path.  The theoretical choice (TH)  directly specifies the value for the regularization parameter and there is  no need for  a time-consuming grid search. For Gaussian errors, this particular TH gives almost comparable performance to SCV-BIC in terms of both prediction and variable selection accuracy.

\section*{Appendix}

\subsection*{Proofs for Section 2}

Throughout this section we will make use of the following basic fact.
\begin{lemma}\label{setb}
For  a given $\alpha \in (0, 1)$,  let  $t = \sqrt{\frac{4\ln (1/\alpha)}{n}} + \frac{4\ln (1/\alpha)}{n}$ and define \begin{equation}\label{bb}\mathcal{B} := \{
\|\epsilon\|_2/\sqrt{n} \leq \sqrt{1 +t}\}.\end{equation}
Then, \[ \mpr\left(\mathcal{B}\right)\geq 1 - \alpha.\]
\end{lemma}
\noindent The proof of this result  is a direct application of  Lemma 8.1 in \cite{Buhlmann11}. Notice that, on $\mathcal{B}$, we have $\|\epsilon\|_2/\sqrt{n} \leq C$, for a dominating constant $C$. We will make implicit use of this fact throughout.

\begin{proof}[Proof of Theorem \ref{thm.yoyo.pred}]
In the first step of the proof, we show that
$\widehat\delta:=\widehat\beta-\beta^0\in\Delta_{\gamma}$. The desired
bounds are then derived in a second step.\\

\noindent For the first step, we note that the definition of the estimator \eqref{eq.yoyo.estimator} implies
\aln{
&\frac{\|Y-X\widehat\beta\|_2}{\sqrt n}-\frac{\|Y-X\beta^0\|_2}{\sqrt n}
\leq \frac{\lambda}{n}\sgr\left(\|\beta^{0j}\|_2-\|\widehat\beta^{j}\|_2\right),
}
and simple algebra yields
\aln{\nonumber
\frac{\lambda}{n}\sgr\left(\|\beta^{0j}\|_2-\|\widehat\beta^{j}\|_2\right)= &\frac{\lambda}{n}\sgrs\left(\|\beta^{0j}\|_2-\|\widehat\beta^{j}\|_2\right)-\frac{\lambda}{n}\sgrsc\|\widehat\beta^{j}\|_2\\\nonumber
\leq &\frac{\lambda}{n}\sgrs|\|\beta^{0j}\|_2-\|\widehat\beta^{j}\|_2|- \frac{\lambda}{n}\sgrsc\|\widehat\beta^{j}\|_2\\
\leq &\frac{\lambda}{n}\sgrs\|\widehat\delta^{j}\|_2- \frac{\lambda}{n}\sgrsc\|\widehat\delta^{j}\|_2.
}
These two inequalities give
\eq{\label{eq.yoyo.basicleft}
\frac{\|Y-X\widehat\beta\|_2}{\sqrt n}\leq \frac{\|Y-X\beta^0\|_2}{\sqrt n}+\frac{\lambda}{n}\sgrs\|\widehat\delta^{j}\|_2- \frac{\lambda}{n}\sgrsc\|\widehat\delta^{j}\|_2.
}
Next, we bound the error term. We obtain, via an application of the
Cauchy-Schwarz's inequality, and recalling the definition of the error term $V$:
\al{\nonumber
| \epsilon'X\widehat\delta|= \ & |\sum_{j=1}^q\epsilon' X^j \widehat\delta ^j|\\
\nonumber\leq& \  \sum_{j=1}^q\|(\epsilon' X^j)'\|_2 \|\widehat\delta ^j\|_2\\
\nonumber\leq& \ \max_{1\leq j \leq q}\left\{\frac{ \sqrt n\|(\epsilon' X^j)'\|_2}{\sqrt T_j\| \epsilon\|_2}\right\}\frac{\| \epsilon\|_2}{\sqrt n} \sgr\|\widehat\delta ^j\|_2\\
=& \ V \ \frac{\| \epsilon\|_2}{\sqrt n} \sgr\|\widehat\delta ^j\|_2.\label{eq.yoyo.emppro}
}
We then observe that
\eqn{\frac{\nabla\|Y-X\beta\|_2|_{\beta=\beta^0}}{\sqrt n}=\frac{\text - X'\epsilon}{\sqrt n\|\epsilon\|_2},
}
and use Inequality \eqref{eq.yoyo.emppro} and the fact that any
norm is convex to obtain
\aln{
\frac{\|Y-X\widehat\beta\|_2}{\sqrt n}-\frac{\|Y-X\beta^0\|_2}{\sqrt n}
\geq & \text - \frac{|\epsilon'X \widehat\delta|}{\sqrt n\|\epsilon\|_2}\\
\geq &\text - \frac{V}{n}\sgr\|\widehat\delta^j\|_2.
}
Since on the set $\mathcal{A}$ we have $\lambda/\overline\gamma\geq V$,  we further obtain
\eq{\label{eq.yoyo.basicright}
\frac{\|Y-X\widehat\beta\|_2}{\sqrt n}-\frac{\|Y-X\beta^0\|_2}{\sqrt n}
\geq \text - \frac{\lambda}{n\overline\gamma}\sgr\|\widehat\delta^j\|_2.
}
Combining \eqref{eq.yoyo.basicleft} and \eqref{eq.yoyo.basicright}, we find
\aln{
 \text - \frac{\lambda}{n\overline\gamma}\sgr\|\widehat\delta^j\|_2\leq \frac{\lambda}{n}\sgrs\|\widehat\delta^{j}\|_2- \frac{\lambda}{n}\sgrsc\|\widehat\delta^{j}\|_2,
}
and thus
\aln{
 \left(1-\frac{1}{\overline\gamma}\right) \frac{\lambda}{n}\sgrsc\|\widehat\delta^j\|_2\leq \left(1+\frac{1}{\overline\gamma}\right)\frac{\lambda}{n}\sgrs\|\widehat\delta^{j}\|_2.
}
This implies $ \frac{\lambda}{n}\sgrsc\|\widehat\delta^j\|_2\leq \left(\frac{\overline\gamma+1}{\overline\gamma-1}\right)\frac{\lambda}{n}\sgrs\|\widehat\delta^{j}\|_2 $ and since $\lambda>0$, we obtain
\al{ \label{eq.cone}
\sgrsc\|\widehat\delta^j\|_2\leq   \gamma\sgrs\|\widehat\delta^{j}\|_2,
}
or  equivalently, $\widehat\delta \in\Delta_{\gamma}$,
as desired.\\

\noindent To derive the bounds stated in the theorem we begin by observing that
\begin{equation}
  \label{eq.yoyo.new}
  \frac{\|Y-X\widehat\beta\|_2}{\sqrt
      n}-\frac{\|Y-X\beta^0\|_2}{\sqrt n} \leq \frac{\lambda}{n}\frac{\sqrt{s^*}\|X\widehat\delta\|_2}{\sqrt n \kappa}
\end{equation}
by  \eqref{eq.yoyo.basicleft} and the Compatibility Condition~\eqref{eq.yoyo.compconst}.
Next, we write
\aln{
\frac{\|Y-X\widehat\beta\|_2^2}{ n}-\frac{\|Y-X\beta^0\|_2^2}{ n}
=&\frac{\|X\widehat\delta -\sigma\epsilon\|_2^2}{ n}-\frac{\|\sigma\epsilon\|_2^2}{ n}
=\frac{\|X\widehat\delta\|_2^2}{ n}-\frac{2\sigma\epsilon'X\widehat\delta}{ n},
}
and we use \eqref{eq.yoyo.compconst}, \eqref{eq.yoyo.emppro}, and  \eqref{eq.yoyo.new}  to obtain
\al{\nonumber
&\frac{\|X\widehat\delta\|_2^2}{ n}
  =\frac{\|Y-X\widehat\beta\|_2^2}{ n}-\frac{\|Y-X\beta^0\|_2^2}{
    n}+\frac{2\sigma\epsilon'X\widehat\delta}{ n}\\
\nonumber
  =&\left(\frac{\|Y-X\widehat\beta\|_2}{\sqrt n}-\frac{\|Y-X\beta^0\|_2}{
    \sqrt n}\right)\left(\frac{\|Y-X\widehat\beta\|_2}{\sqrt n}+\frac{\|Y-X\beta^0\|_2}{
    \sqrt n}\right)+\frac{2\sigma\epsilon'X\widehat\delta}{ n}\\
 \nonumber
  \leq& \frac{\lambda}{n}\frac{\sqrt{s^*}\|X\widehat\delta\|_2}{\sqrt n \kappa}\left(\frac{2\|Y-X\beta^0\|_2}{\sqrt
      n}+\frac{\lambda}{n}\frac{\sqrt{s^*}\|X\widehat\delta\|_2}{\sqrt n \kappa}\right)
+\frac{2V\|\sigma\epsilon\|_2}{n^{{3}/{2}}}\sgr\|\widehat\delta^j\|_2
\\
 \nonumber
  \leq& \frac{s^*\lambda^2}{ \kappa^2n^2}\frac{\|X\widehat\delta\|_2^2}{n}+\frac{2\lambda}{n}\frac{\sqrt{s^*}\|X\widehat\delta\|_2}{\sqrt n \kappa}\frac{\|\sigma\epsilon\|_2}{\sqrt
      n}
+\frac{2(1+\gamma)\lambda\|\sigma\epsilon\|_2}{\overline{\gamma}n^{{3}/{2}}}\frac{\sqrt{s^*}\|X\widehat\delta\|_2}{\sqrt n \kappa}\\
 \nonumber
  =& \frac{s^*\lambda^2}{ \kappa^2n^2}\frac{\|X\widehat\delta\|_2^2}{n}+\gamma\frac{2\lambda}{n}\frac{\sqrt{s^*}\|X\widehat\delta\|_2}{\sqrt n \kappa}\frac{\|\sigma\epsilon\|_2}{\sqrt
      n}
}
since on $\mathcal{A}$ we have $\lambda/\overline\gamma\geq V$.
Consequently,
\[
\frac{\|X\widehat\delta\|_2^2}{n}\leq  \frac{u\sqrt{s^*}\lambda\|\sigma\epsilon\|_2\|X\widehat\delta\|_2}{n^2\kappa}, \]
where  $u:= \frac{2\gamma}{1-\frac{\lambda^2s^*}{n^2\kappa^2}}\in(0,\infty)$ by assumption (\ref{less}). Since on $\mathcal{B}$ we have
\[ \frac{\|\sigma\epsilon\|_2}{\sqrt n} \leq\sigma \sqrt{1 + t}, \]  the first statement of the theorem follows:
\eqn{
\|X(\widehat\beta-\beta^0)\|_2\leq \sigma \sqrt{1 + t} \frac{\lambda\sqrt{s^*}u}{\kappa\sqrt n}\lesssim \sqrt{1 + t}  \frac{\lambda\sqrt{s^*}}{\kappa\sqrt n}.\
}
For the second claim, we use the fact that
$\delta\in\Delta_{\gamma}$ and the Compatibility Condition \eqref{eq.yoyo.compconst} to
deduce that
\eqn{
  \sgr\|\delta^j\|_2\leq (\gamma+1)\sgrs\|\delta^j\|_2\leq \frac{(\gamma+1)\sqrt{s^*}\|X(\widehat\beta-\beta^0)\|_2}{\sqrt n \kappa}\leq  \frac{\lambda (\gamma+1)us^*}{n\kappa^2}\frac{\|\sigma\epsilon\|_2}{\sqrt n}.
}
Therefore, again on the set $\mathcal{B}$, we have

\eqn{
  \sgr\|(\hat{\beta} - \beta)^j\|_2 \leq \sigma \sqrt{1 + t}\frac{\lambda (\gamma+1)us^*}{n\kappa^2}\lesssim \sqrt{1 + t}\frac{\lambda s^*}{n\kappa^2},
  }
  which concludes the proof of this theorem.
\end{proof}


\noindent We  next prove Theorem  \ref{thm.yoyo.varsel}.  We begin with  two  preparatory results.

\begin{lemma}\label{lemma.yoyo.kkt}
Assume that  $Y\neq X\widehat\beta$ over some set $\mathcal{C}$ . Then,  on $\mathcal{C}$, the quantity $\widehat\beta$ is a solution of the criterion \eqref{eq.yoyo.estimator} if and only if for every  $1\leq j\leq q$
\al{
&\widehat\beta^j\neq 0 \Rightarrow \frac{(X'(Y-X\widehat\beta))^j}{\|Y-X\widehat\beta\|_2}=\frac{\lambda\sqrt{T_j}}{\sqrt n\|\widehat\beta^j\|_2}\widehat\beta^j \label{kkt1}\\
&\widehat\beta^j= 0\Rightarrow \frac{\|(X'(Y-X\widehat\beta))^j\|_2}{\|Y-X\widehat\beta\|_2}\leq \frac{\lambda\sqrt{T_j}}{\sqrt n}.  \label{kkt2}
}
\end{lemma}
\begin{proof}
Since all terms of the criterion \eqref{eq.yoyo.estimator} are convex, and
thus, the criterion is convex, we can apply standard subgradient calculus. The subgradient $\partial f|_x$ of a convex function $f:\mr^p\to \mr$ at a
point $x\in\mr^p$ is defined as the
set of vectors $v\in\mr^p$ such that for all $y\in\mr^p$
\begin{equation*}
  f(y)\geq f(x)+v'(y-x).
\end{equation*}
From this, one derives easily that subgradients are linear and additive and
that the subgradient $\partial f|_x$ is equal to the gradient $\nabla f|_x$ if
the function $f$ is differentiable at $x$. Moreover,  $x\in\mr^p$ is a
minimum of the function $f$ if  and only if $0\in\partial f|_x$.
Since $Y\neq X\widehat\beta$, the first term of the criterion \eqref{eq.yoyo.estimator} is differentiable and we have
\al{\label{eq.yoyo.sub1}
\nabla\|Y-X\beta\|_2|_\beta=&\frac{\nabla\|Y-X\beta\|_2^2|_\beta}{2\|Y-X\beta\|_2}=\frac{\text -X'(Y-X\beta)}{\|Y-X\beta\|_2}.
}
For the remaining terms, we observe
that for any vector $w\in\mr^T\backslash\{0\}$, $T\in\mn$,
\al{\label{eq.yoyo.sub2}
\nabla\|w\|_2|_w=\frac{w~}{\|w\|_2}
}
and for $w=0$
\al{\label{eq.yoyo.sub3}
v\in\partial\|w\|_2|_{w=0}\Leftrightarrow\|z\|_2\geq\|0\|_2+(z-0)'v=z'v~~~\text{for
all~}z\in\mr^T
}
and, consequently, $\partial \|w\|_2|_{w=0}=\{v\in\mr^T:\|v\|_2\leq 1\}$.\\
The claim follows then from Equations \eqref{eq.yoyo.sub1},
\eqref{eq.yoyo.sub2}, and  \eqref{eq.yoyo.sub3}.
\end{proof}

\begin{lemma}\label{lemma.yoyo.tech} Under the conditions of
  Theorem~\ref{thm.yoyo.pred}, it holds that, on the set $\mathcal{A} \cap \mathcal{B}$, we have
\eqn{
\left(1-\frac{\lambda\sqrt{s^*}u}{n\kappa}\right)\|\sigma\epsilon\|_2\leq \|Y-X\widehat\beta\|_2\leq  \left(1+\frac{\lambda\sqrt{s^*}u}{n\kappa}\right)\|\sigma\epsilon\|_2
}
for $u:=\frac{2\gamma}{1-\frac{\lambda^2s^*}{n^2\kappa^2}}$.
\end{lemma}
\begin{proof}
By  the triangle inequality
\eqn{
\|\sigma\epsilon\|_2-\|X(\widehat\beta-\beta_0)\|_2\leq\|Y-X\widehat\beta\|_2\leq\|\sigma\epsilon\|_2+\|X(\widehat\beta-\beta_0)\|_2.
}
The claim follows immediately  by Theorem~\ref{thm.yoyo.pred} above.
\end{proof}

\begin{proof}[Proof of Theorem \ref{thm.yoyo.varsel}]
The crucial step  in  this proof is to use the KKT Conditions in
Lemma~\ref{lemma.yoyo.kkt} in order to show that, on $\mathcal{A}_1 \cap \mathcal{B}$,  we have $\widehat\beta_{S^c}=0$.\\

First, we observe that Lemma~\ref{lemma.yoyo.tech} implies that
$Y-X\widehat\beta\neq 0$ on $\mathcal{A} \cap \mathcal{B}$, and we can consequently apply the KKT Conditions derived in
Lemma~\ref{lemma.yoyo.kkt} for $\mathcal{C} = \mathcal{A} \cap \mathcal{B}$.  Moreover,  since  by definition,  $\mathcal{A}_1 \cap \mathcal{B} \subseteq \mathcal{A} \cap \mathcal{B}$,  the results also hold on the smaller set. Thus, there exists  a
 vector $\tau\in\mr^p$ such that $\|\tau^j\|_2\leq \sqrt{T_j}$ for all
 $1\leq j \leq q$ and, additionally,
 $\tau^j=\frac{\sqrt{T_j}\widehat\beta^j}{\|\widehat\beta^j\|_2}$,  for all
 $1\leq j \leq q$ such that $\widehat\beta^j\neq 0$, and $\tau$ satisfies
 the equality
\eqn{
\frac{X'(Y-X\widehat\beta)}{\|Y-X\widehat\beta\|_2}=\frac{\lambda}{\sqrt n}\tau.
}

\noindent  We rewrite this with  \[\widehat \psi:=\|Y-X\widehat\beta\|_2 \ \mbox{and~}
\widehat\delta:=\widehat\beta-\beta^0 \] as
\eqn{
\sigma X'\epsilon-X'X\widehat\delta=\frac{\widehat \psi\lambda}{\sqrt n}\tau.
}
So, on the one hand, we have
\eq{\label{eq.yoyo.onehand}
\text- n^2\See \widehat\delta_S-n^2\Sez \widehat\delta_{S^c} = \sqrt n\widehat \psi\lambda \tau_S-n\sigma (X'\epsilon)_S,
}
or, equivalently,
\eqn{
\text-n^2\widehat\delta_S- n^2\Seem\Sez \widehat\delta_{S^c} = \sqrt n\widehat \psi\lambda  \Seem\tau_S-n\sigma  \Seem(X'\epsilon)_S,
}
and finally
\eq{\label{eq.yoyo.vsfirst}
\text-n^2\widehat\delta_{S^c}'\Sze\widehat\delta_S-n^2\widehat\delta_{S^c}'\Sze \Seem\Sez \widehat\delta_{S^c} = \sqrt n\widehat \psi\lambda \widehat\delta_{S^c}'\Sze \Seem\tau_S-n\sigma \widehat\delta_{S^c}'\Sze \Seem(X'\epsilon)_S.
}
On the other hand, we have
\eqn{
\text-n^2\Sze \widehat\delta_S-n^2\Szz \widehat\delta_{S^c} = \sqrt n\widehat \psi\lambda \tau_{S^c}-n\sigma (X'\epsilon)_{S^c}.
}
Since for $j\in S^c$
\aln{
\widehat\beta^j\neq 0&\Rightarrow\widehat\delta^j\cdot\tau^j=\frac{\sqrt{T_j}~\widehat\delta^j\cdot\widehat\delta^j}{\|\widehat\delta\|_2}=\sqrt{T_j}\|\widehat\delta^j\|_2\\
\widehat\beta^j= 0&\Rightarrow\widehat\delta^j\cdot\tau^j=0=\sqrt{T_j}\|\widehat\delta^j\|_2,
}
this implies that
\aln{
  \text-n^2\widehat\delta_{S^c}'\Sze \widehat\delta_S-n^2\widehat\delta_{S^c}'\Szz \widehat\delta_{S^c} &= \sqrt n\widehat \psi\lambda\widehat\delta_{S^c}'\tau_{S^c}-n\sigma \widehat\delta_{S^c}'(X'\epsilon)_{S^c}\\
  &= \sqrt n\widehat \psi\lambda\sgrsc\left(\|\widehat\delta^j\|_2-\frac{\sigma\sqrt n~\widehat\delta^j\cdot (X'\epsilon)^j}{\sqrt{T_j}\lambda\widehat \psi} \right).
}
The right-hand side can be bounded from below, using Cauchy-Schwarz's Inequality, by
\aln{
&\sqrt n\widehat \psi\lambda\sgrsc\left(\|\widehat\delta^j\|_2-\|\widehat\delta^j\|_2\frac{\sigma\sqrt n\|(X'\epsilon)^j\|_2}{\sqrt{T_j}\lambda \widehat \psi}\right).
}
Lemma~\ref{lemma.yoyo.tech} implies that $\lambda/\widetilde\eta\geq \widehat V$ for
\eqn{
\widehat V:=\max_{1\leq j \leq q}\left\{\frac{\sigma \sqrt n\|(X'\epsilon)^j\|_2}{\sqrt{T_j}\widehat \psi}\right\},
}
and thus, the above term can be bounded from below by
\eqn{
\left(1-\frac{1}{{\widetilde\eta}}\right)\sqrt n\widehat \psi\lambda \sgrsc\|\widehat\delta^j\|_2.
}
So, in summary, we have
\eq{\label{eq.yoyo.vssecond}
\text-n^2\widehat\delta_{S^c}'\Sze \widehat\delta_S-n^2\widehat\delta_{S^c}'\Szz \widehat\delta_{S^c}\geq \left(1-\frac{1}{{\widetilde\eta}}\right)\sqrt n\widehat \psi\lambda \sgrsc\|\widehat\delta^j\|_2.
}
Subtracting Equation~\eqref{eq.yoyo.vssecond} from
Equation~\eqref{eq.yoyo.vsfirst} then yields
\begin{eqnarray} \label{eq.yoyo.varsel2}
&& n^2\widehat\delta_{S^c}'(\Szz-\Sze \Seem\Sez) \widehat\delta_{S^c}  \nonumber  \\
&& \leq  \sqrt n\widehat \psi\lambda \widehat\delta_{S^c}'\Sze \Seem(\tau_S-\frac{\sqrt n\sigma}{\lambda \widehat \psi}(X'\epsilon)_S)-\left(1-\frac{1}{{\widetilde\eta}}\right)\sqrt n\widehat \psi\lambda \sgrsc\|\widehat\delta^j\|_2. \end{eqnarray}
The first term of the right-hand side above can be bounded via the Cauchy-Schwarz's
inequality by
\aln{
\sqrt n\widehat \psi\lambda \widehat\delta_{S^c}'\Sze \Seem(\tau_S-\frac{\sqrt n\sigma}{\lambda \widehat \psi}(X'\epsilon)_S)&=\sqrt n\widehat \psi\lambda \sum_{j\in S^c}\widehat\delta^j\cdot(\Szet \Seem(\tau_S-\frac{\sqrt n\sigma}{\lambda \widehat \psi}(X'\epsilon)_S))^j\\
&\leq
\sqrt n\widehat \psi\lambda \sgrsc\|\widehat\delta^j\|_2\frac{\|(\Szet \Seem(\tau_S-\frac{\sqrt n\sigma}{\lambda \widehat \psi}(X'\epsilon)_S))^j\|_2}{\sqrt{T_j}}.
}
Now, we observe that if $\lambda/{\widetilde\eta}\geq\widehat V$, then $\frac{\sigma\sqrt n}{\widehat \psi\lambda}\|(X'\epsilon)^j\|_2\leq \frac{\sqrt{T_j}}{{\widetilde\eta}}$ for all $0\leq j\leq q$, and thus, the above expression can be bounded by
\aln{ &\sqrt n\widehat \psi\lambda \max_{v:\|v^k\|_2\leq \left(1+\frac 1 {\widetilde\eta}\right) \sqrt{T_k}}\sgrsc\|\widehat\delta^j\|_2\frac{\|(\Szet \Seem v)^j\|_2}{\sqrt{T_j}}\\
=&\left(1+\frac 1 {\widetilde\eta}\right)\sqrt n\widehat \psi\lambda \max_{v:\|v^k\|_2\leq \sqrt{T_k}}\sgrsc\|\widehat\delta^j\|_2\frac{\|(\Szet \Seem v)^j\|_2}{\sqrt{T_j}}.
}
If $\widehat\beta_{S^c}\neq 0$, this is strictly smaller than
\aln{
\left(1+\frac 1 {\widetilde\eta}\right)u\sqrt n\widehat \psi\lambda\sgrsc\|\widehat\delta^j\|_2
=\left(1-\frac 1 {\widetilde\eta}\right)\sqrt n\widehat \psi\lambda\sgrsc\|\widehat\delta^j\|_2
}
by  our  Group Irrepresentable Condition. Then, by  Inequality~\eqref{eq.yoyo.varsel2}, this yields
\aln{
 n^2\widehat\delta_{S^c}'(\Szz-\Sze \Seem\Sez) \widehat\delta_{S^c}<0.
}
But since $\Szz-\Sze \Seem\Sez\geq 0$, this leads to a
contradiction. Hence, $\widehat\delta_{S^c}=0$ and the first claim is proved.\\

\noindent For the second claim, we invoke $\widehat\delta_{S^c}=0$ to obtain, using
Equation~\eqref{eq.yoyo.onehand},
\eqn{
\text- n^2\See \widehat\delta_S = \sqrt n\widehat \psi\lambda  \tau_S-n\sigma (X'\epsilon)_S.
}
This implies
\eqn{
\text- n^2\widehat\delta_S = \sqrt n\widehat \psi\lambda   \Seem\left(\tau_S-\frac{\sqrt n\sigma (X'\epsilon)_S}{\widehat \psi \lambda}\right)
}
and, using $\lambda/{\widetilde\eta}\leq \widehat V$ and bounding the norms as above,
\aln{
  \|\widehat\delta^j\|_{\infty}\leq&\max_{v:\|v^k\|_2\leq \sqrt{T_j}} \frac{\left(1+\frac{1}{{\widetilde\eta}} \right)\lambda}{n}\frac{\widehat \psi}{\sqrt n} \|(\Seemt v)^j\|_{\infty}\\
\leq& \  \ \frac{\left(1+\frac{1}{{\widetilde\eta}} \right)\sqrt{T_j}\lambda}{n}\frac{\widehat \psi}{\sqrt n}\xi_{\|\cdot \|_{\infty}} \\
\leq&  \      \frac{\left(1+\frac{1}{{\widetilde\eta}} \right)\sqrt{T_j}\lambda}{n}\frac{\|\sigma\epsilon\|_2}{\sqrt n}\left(1+\frac{\lambda\sqrt{s^*}u}{n\kappa}\right)\xi_{\|\cdot \|_{\infty}} \\
\leq& \ \frac{2\sigma \sqrt{1 + t}}{1 + \widetilde\eta}(1 + u) \xi_{\|\cdot \|_{\infty}} \frac{\lambda\sqrt{T_j}}{n}    \\
\leq  & \ D \frac{\lambda\sqrt{T_j}}{n},  \ \mbox{for all} \ 1\leq j \leq q,
}
which  is the second claim of this theorem.  In the above derivation we used Lemma ~\ref{lemma.yoyo.tech}  above for the third inequality,  and assumption (\ref{less}) for the forth and the fact that  $\sqrt{1 + t}$ is bounded by a constant, by the definition of the set $\mathcal{B}$ in Lemma \ref{setb} above. We also recall that under  (\ref{less}), the quantity $u$ is a positive constant.

\noindent The third claim follows immediately from the first two and the Beta Min Condition.  This concludes the prof of this theorem.
\end{proof}

 \begin{proof}[Proof of Lemma \ref{lambdagauss}] 

 We first observe that
\begin{align*}
\mpr\left(V\geq \lambda_0\right)=& \ \mpr\left(\max_{1\leq j \leq q}\left\{\frac{ \sqrt n\|(\epsilon' X^j)'\|_2}{\sqrt T_j\| \epsilon\|_2}\right\}\geq \lambda_0\right)\\
= & \ \mpr\left(\max_{1\leq j \leq q} \epsilon' \left( \frac{X^j (X^j)'}{n} - \frac{\lambda_0^2 T_j}{n^2} I  \right)\epsilon \geq 0\right)\\
\leq& \  \sum_{j=1}^q \mpr\left(\epsilon' \left( \frac{X^j (X^j)'}{n} - \frac{\lambda_0^2 T_j}{n^2} I  \right)\epsilon \geq 0\right).
\end{align*}
Let $U'(j) D(j) U(j)$ be a spectral decomposition of ${X^j (X^j)'}/{n}$ such that $U(j)$ is orthogonal and $D(j)$ is diagonal with diagonal entries $\xi_1(j)\geq\dots\geq \xi_{T_j}(j)\geq\xi_{T_j+1}(j)=\dots =\xi_{n}(j)=0$.  With the notation $\zeta_j= \|X^j\|^2 /n$,  where $\| A \|$ is the spectral norm of a generic matrix $A$, we have  $\zeta_j = \xi_1(j)$.  It follows that
\begin{align*}
  \epsilon' \left( \frac{X^j (X^j)'}{n} - \frac{\lambda_0^2 T_j}{n^2} I  \right)\epsilon=&\epsilon' \left( {U'(j) D(j) U(j)} - \frac{\lambda_0^2 T_j}{n^2} I  \right)\epsilon\\
=&(U(j)\epsilon)' \left( D(j) - \frac{\lambda_0^2 T_j}{n^2} I  \right)U(j)\epsilon\\
\leq&  \|\epsilon_1\|_2^2 \left( \zeta_j- \frac{\lambda_0^2 T_j}{n^2} \right)_+ - \|\epsilon_2\|_2^2 \frac{\lambda_0^2 T_j}{n^2},
\end{align*}
where $\epsilon_1$ and $\epsilon_2$ are independent with $\epsilon_1 \sim \mathcal N(0, I_{T_j})$ and $\epsilon_2 \sim \mathcal N(0, I_{(n-T_j)})$. Thus, for any fixed $r \in (0, 1)$ we have
\begin{eqnarray} \label{xibound}
\mpr\left(V\geq \lambda_0\right)
&\leq & \ \sum_{j=1}^{q} \mpr\left(\frac{ \|\epsilon_1\|_2^2}{T_j } \cdot ( \zeta_j - \frac{\lambda_0^2T_j}{n^2})_+ \geq  \frac{\lambda_0^2 }{n^2} \cdot \|\epsilon_2\|_2^2 \right)  \\
&\leq &  \ \sum_{j=1}^{q} \  \mpr\left(\frac{ \|\epsilon_1\|_2^2}{ T_j}  \cdot (\frac{ n^2\zeta_j}{\lambda_0^2} - T_j )_+\cdot \frac{1}{n-T_j}\geq  1 - r \right) \nonumber \\
&+ &  \ \sum_{j=1}^{q} \ \mpr\left(       \frac{1}{n-T_j} \cdot \|\epsilon_2\|_2^2  \leq 1 - r  \right). \nonumber
\end{eqnarray}
If $ \zeta_j \leq  \frac{\lambda_0^2T_j}{n^2}$, then the first sum  in the inequality above is trivially equal to zero, therefore the argument below is needed only when the reverse inequality holds.  From Laurent and Massart \cite[Lemma 1]{lm2000}, $\mpr (X - d \geq d t ) \leq \exp\left(-\frac{d}{4} (\sqrt { (1 + 2 t} - 1)^2\right)$ and $P(X \leq d - dt ) \leq \exp \left(- \frac{d}{4} t^2\right)$,  for $X\sim \chi^2(d)$.
Therefore, for the first term in (\ref{xibound}) we obtain, for each $j$:
\begin{eqnarray*}
&&  \mpr\left(\frac{ \|\epsilon_1\|_2^2}{ T_j}  \cdot (\frac{ n^2\zeta_j}{\lambda_0^2} - T_j )\cdot \frac{1}{n-T_j}\geq  1 - r \right)   \\
&& \leq    \exp\left( - \frac{T_j } {4} \left(\sqrt{ \frac{2(1-r) (n - T_j)}{\frac{n^2 \zeta_j}{\lambda_0^2} - T_j} - 1} - 1\right)^2 \right)\\
 & &\leq   \exp\left( - \frac{T_{\min} } {4} \left(\sqrt{ \frac{2(1-r) (n - T_{\max})}{\frac{n^2 \zeta}{\lambda_0^2} - T_{\min}} - 1} - 1\right)^2 \right).
\end{eqnarray*}
To bound the last  term in (\ref{xibound})   
we  first obtain, for each $j$:
\eqn{
\mpr\left(\frac{\|\epsilon\|_2^2}{n - T_j}< 1-r\right)\leq \exp\left({-\frac{(n-T_j)r^2}{4}}\right) \leq \exp\left({-\frac{(n-T_{\max})r^2}{4}}\right).
}
Hence,
\begin{eqnarray*}
&& \mpr\left(V\geq \lambda_0\right) \\
&&\leq  q\cdot \exp\left( - \frac{T_{\min} } {4} \left(\sqrt{ \frac{2(1-r) (n - T_{\max})}{\frac{n^2 \zeta}{\lambda_0^2} - T_{\min}} - 1} - 1\right)^2 \right) + q \cdot \exp\left({-\frac{(n-T_{\max})r^2}{4}}\right).
\end{eqnarray*}
For $r = 2\sqrt{\frac{\log(2q/\alpha)}{n-T_\text{max}}}$ the last term is bounded by $\alpha/2$.  For this value of $r$ and with

\[ \lambda_0 = \frac{\sqrt{2\zeta}n}{\sqrt{n-T_\text{max}}}\left(1+\sqrt{\frac{2\log(2q/\alpha)}{T_\text{min}}}\right),\]
the first term is also bounded by $\alpha/2$. This concludes the proof.

\end{proof}


\subsection*{Proofs for Section~\ref{sec:comp}}

\begin{lemma}\label{nonexp}
Given any $\lambda$, $\vec\Theta(\cdot;\lambda)$ is nonexpansive: $\| \vec\Theta(x; \lambda) - \vec\Theta(\tilde x; \lambda) \|_2 \leq \|x - \tilde x\|_2$, $\forall x, \tilde x\in {\mathbb R}^{p}$.
\end{lemma}

\begin{proof}
Define  $\Delta = \|x - \tilde x\|_2^2 -  \|\vec \Theta( x; \lambda) - \vec \Theta( \tilde x; \lambda)\|_2^2$,   $a=\|x\|_2$, $b=\|\tilde x\|_2$, and $c = x' \tilde x/(ab)$. Obviously, $|c|\leq 1$ and  $c = \vec \Theta( x; \lambda)' \vec \Theta( \tilde x; \lambda) /(ab)$.
By the Cosine Rule,
\begin{align*}
\| x - \tilde x \|_2^2 & = a^2 + b^2 - 2abc\\
\| \vec \Theta( x; \lambda) - \vec \Theta( \tilde x; \lambda) \|_2^2 & = ((a-\lambda)_+)^2 + ((b-\lambda)_+)^2 - 2(a-\lambda)_+ (b-\lambda)_+ c.
\end{align*}

 (i) Suppose $a>\lambda$ and $b>\lambda$. Then $\Delta = -2\lambda^2 +2(a+b)\lambda  +2\lambda^2 c - 2\lambda(a+b) c = 2(1-c) \lambda (a+b-\lambda)\geq 0$.

(ii) Suppose $a< \lambda$ and  $b>\lambda$. Then $\Delta = a^2 + b^2 - 2 abc - (b-\lambda)^2 = a^2 - 2 abc - \lambda^2 + 2b \lambda \geq a^2 - 2 ab  - \lambda^2 + 2b \lambda = (2b - a -\lambda) (\lambda-a) \geq 0$. Therefore, $\|\vec \Theta( x; \lambda) - \vec \Theta( \tilde x; \lambda)\|_2 \leq \|x - \tilde x\|_2$.
\end{proof}

\begin{proof}[Proof of Theorem \ref{th:stisp-conv}]
By Lemma \ref{nonexp},  the mapping \eqref{tisp-iter} is nonexpansive. We use  Opial's conditions \cite{Opial67,She10} for studying nonexpansive operators to prove the strict convergence of $\beta{(t)}$.
The key of the proof is  to show the mapping  is \emph{asymptotically regular}: $\| \beta{(t+1)} -\beta{(t)} \| \rightarrow 0$
as $t\rightarrow \infty$, for any starting point $\beta{(0)}$.

Assume the scaling operations \eqref{scaling} have performed beforehand.
 Let $F(\beta)=\|Y - X \beta\|_2 + \sum_{j=1}^q \lambda_j \|\beta^j\|_2$ be the objective function.
Introduce a surrogate function
\begin{eqnarray}
G(\beta, \gamma) & =& \| Y - X \beta \|_2 + \frac{1}{\| X \beta - Y\|_2} (\gamma - \beta)' X' (X \beta - Y) \notag \\ & & + \frac{1}{2\| X \beta - Y\|_2} \| \beta - \gamma\|_2^2 + \sum_j \lambda_j  \|\gamma^j\|_2.
\end{eqnarray}

Given $\beta$,  algebraic manipulations   show that minimizing $G$ over $\gamma$ is equivalent to
\begin{eqnarray}
&&\min_\gamma   \frac{1}{\| X \beta - Y\|_2} (\gamma - \beta)' X' (X \beta - Y) + \frac{1}{2\| X \beta - Y\|_2} \| \beta - \gamma\|_2^2 + \sum_j \lambda_j  \|\gamma^j\|_2
\Longleftrightarrow \notag\\
 && \min_\gamma \frac{1}{\| X \beta - Y\|_2} \left( \frac{1}{2} \left\|\gamma - \left [\beta+X' Y-X' X \beta\right ]\right\|_2^2 +  \| X \beta - Y\|_2 \sum_j \lambda_j  \|\gamma^j\|_2\right). \label{optovergamma}
\end{eqnarray}

Applying Lemma 1 and Lemma 2 in \cite{She12}, we have the optimal $\gamma_o$  given by
\begin{align}
\gamma_o^j  = \vec \Theta( \beta^j  + (X^j)'  Y -    (X^j)' X \beta   ;  \| X \beta  - Y\|_2 \lambda_j), \quad 1\leq j \leq q,
\end{align}
and  further obtain
\begin{align*}
G(\beta, \gamma_o + \delta) - G(\beta, \gamma_o)  \geq  \frac{\|\delta\|_2^2}{2 \| X \beta - Y\|_2}.
\end{align*}

On the other hand, a Taylor series expansion gives
\begin{align*}
&\| Y - X \beta\|_2 +
\frac{1}{\| X \beta - Y\|_2} (\gamma - \beta)' X' (X \beta - Y)   - \| Y - X \gamma \|_2\\
= & - \frac{1}{2} (\beta - \gamma)' \left[ \frac{1}{\| X \xi - Y\|_2} X' X - \frac{1}{\| X \beta - Y\|_2^3} X' (X\xi - Y) (X \xi - Y)' X\right] (\beta - \gamma),
\end{align*}
for some $\xi=\vartheta\beta+(1-\vartheta) \beta$ with $\vartheta \in (0, 1)$.

Now, for the iterates defined by \eqref{tisp-iter}, we obtain
\begin{eqnarray*}
&&F(\beta{(t+1)})+\frac{1}{2}(\beta{(t+1)}-\beta{(t)})' \left(\frac{1}{\| X \beta{(t)}  - Y\|_2} I - \frac{1}{\| X \xi{(t)}  - Y\|_2} X' X \right)(\beta{(t+1)}-\beta{(t)}) \\
&=&  G(\beta{(t)},\beta{(t+1)})\\
 &\leq& G(\beta{(t)},\beta{(t)})-\frac{1}{2} \frac{1}{  \| X \beta{(t)}  - Y\|_2} (\beta{(t+1)}-\beta{(t)})'  (\beta{(t+1)}-\beta{(t)})\\
&=&F(\beta{(t)})-\frac{1}{2} \frac{1}{  \| X \beta{(t)}  - Y\|_2}  (\beta{(t+1)}-\beta{(t)})'  (\beta{(t+1)}-\beta{(t)}),
\end{eqnarray*}
for some $\xi{(t)}=\vartheta{(t)}\beta{(t)}+(1-\vartheta{(t)}) \beta{(t+1)}$ with $\vartheta{(t)} \in (0, 1)$. Therefore, with $\| X\|$ standing for the operator norm of $X$,
\begin{eqnarray}
F(\beta{(t)}) - F(\beta{(t+1)}) \geq  \frac{1}{2} \left(  \frac{2}{  \| X \beta{(t)}  - Y\|_2} - \frac{\|  X\|^2}{\| X \xi{(t)}  - Y\|_2}  \right) \| \beta{(t+1)}-\beta{(t)} \|_2^2. \label{errbnd}
\end{eqnarray}
Under the regularity condition and  for $K$ large enough, $F(\beta{(t)})$ is monotonically decreasing. 
In fact, with $\|X\xi(t) - Y\|_2>\epsilon$ and  $M \triangleq F(\beta(0))$, $\| X\|_2 < 2\epsilon /M$ suffices. 
It follows that    
\begin{align*}
F(\beta(t+1)) \leq F(\beta(t)) \leq M, \forall t, \mbox{ and } \left({\frac{2}{  \| X \beta{(t)}  - Y\|_2} - \frac{\|  X\|^2}{\| X \xi{(t)}  - Y\|_2}  }\right) \| \beta{(t+1)}-\beta{(t)} \|_2^2 \rightarrow 0 \mbox{ as } t\rightarrow \infty. 
\end{align*}
This, together with the  conditions in the theorem, implies that    $\beta(t)$ is uniformly bounded and asymptotically regular. 
Finally, the fixed point set of the mapping is non-empty because it is a nonexpansive mapping into a bounded closed convex subset \cite{browder1965nonexpansive}. 

 With all of Optial's conditions  satisfied, $\beta(t)$ has a unique limit point $\beta^*$. It is easy to verify that $\beta^*$ as a fixed point of \eqref{tisp-iter} satisfies the KKT equations \eqref{kkt1} and \eqref{kkt2}. This means $\beta^*$ is a global minimum.
\end{proof}

\subsection*{Acknowledgements}
We thank the associate editor and the referees for their detailed and helpful comments.

\bibliography{Literature4}

\begin{thebibliography}{10}

\bibitem{Bellarx}
A.~Belloni, V.~Chernozhukov, and L.~Wang.
\newblock Pivotal estimation of nonparametric functions via square-root lasso.
\newblock arXiv:1105.1475.

\bibitem{Belloni11}
A.~Belloni, V.~Chernozhukov, and L.~Wang.
\newblock Square-root lasso: pivotal recovery of sparse signals via conic
  programming.
\newblock {\em Biometrika}, 98(4):791--806, 2011.

\bibitem{Bickel09}
P.~Bickel, Y.~Ritov, and A.~Tsybakov.
\newblock Simultaneous analysis of lasso and {D}antzig selector.
\newblock {\em Ann. Statist.}, 37(4):1705--1732, 2009.

\bibitem{Bousquet02}
O.~Bousquet.
\newblock A {B}ennett concentration inequality and its application to suprema
  of empirical processes.
\newblock {\em C. R. Math. Acad. Sci. Paris}, 334(6):495--500, 2002.

\bibitem{browder1965nonexpansive}
F.~E. Browder.
\newblock Nonexpansive nonlinear operators in a banach space.
\newblock {\em Proceedings of the National Academy of Sciences of the United
  States of America}, 54(4):1041, 1965.

\bibitem{Buhlmann11}
P.~B{\"u}hlmann and S.~van~de Geer.
\newblock {\em Statistics for high-dimensional data}.
\newblock Springer Series in Statistics. Springer, 2011.
\newblock Methods, theory and applications.

\bibitem{BuneaEN}
F.~Bunea.
\newblock Honest variable selection in linear and logistic regression models
  via {$\ell_1$} and {$\ell_1+\ell_2$} penalization.
\newblock {\em Electron. J. Stat.}, 2:1153--1194, 2008.

\bibitem{bb}
F.~Bunea and A.~Barbu.
\newblock Dimension reduction and variable selection in case-control studies
  via regularized likelihood optimization.
\newblock {\em Electron. J. Stat}, 3:1257--1287, 2008.

\bibitem{FHHT07}
J.~Friedman, T.~Hastie, H.~H{\"o}fling, and R.~Tibshirani.
\newblock Pathwise coordinate optimization.
\newblock {\em Ann. Appl. Stat.}, 1(2):302--332, 2007.

\bibitem{MomoYoyo12}
M.~Hebiri and J.~Lederer.
\newblock How correlations influence {L}asso prediction.
\newblock {\em IEEE Trans. Inform. Theory}, 59(3):1846--1854, 2013.

\bibitem{KY10}
V.~Koltchinskii and M.~Yuan.
\newblock Sparsity in multiple kernel learning.
\newblock {\em Ann. Statist.}, 38:3660--3695, 2010.

\bibitem{lm2000}
B.~Laurent and P.~Massart.
\newblock Adaptive estimation of a quadratic functional by model selection.
\newblock {\em Ann. Statist.}, 28(5):1302--1338, 2000.

\bibitem{YoyoSara12}
J.~Lederer and S.~van~de Geer.
\newblock New concentration inequalities for empirical processes.
\newblock {\em to appear in Bernoulli}, 2011.

\bibitem{lounici11}
K.~Lounici, M.~Pontil, A.~B. Tsybakov, and S.~{van de Geer}.
\newblock Oracle inequalities and optimal inference under group sparsity.
\newblock {\em Annals of Statistics}, 39:2164--2204, 2011.

\bibitem{Meier9}
L.~Meier, S.~van~de Geer, and P.~Buhlmann.
\newblock High-dimensional additive modeling.
\newblock {\em Ann. Statist.}, 37:3779--3821, 2009.

\bibitem{Meinshausen06}
N.~Meinshausen and P.~B{\"u}hlmann.
\newblock High-dimensional graphs and variable selection with the lasso.
\newblock {\em Ann. Statist.}, 34(3):1436--1462, 2006.

\bibitem{Opial67}
Z.~Opial.
\newblock Weak convergence of the sequence of successive approximations for
  nonexpansive mappings.
\newblock {\em Bull. Amer. Math. Soc.}, 73:591--597, 1967.

\bibitem{Owen07}
A.~Owen.
\newblock A robust hybrid of lasso and ridge regression.
\newblock In {\em Prediction and discovery}, volume 443 of {\em Contemp.
  Math.}, pages 59--71. Amer. Math. Soc., Providence, RI, 2007.

\bibitem{Schelldorfer11}
J.~Schelldorfer, P.~B{\"u}hlmann, and S.~van~de Geer.
\newblock Estimation for high-dimensional linear mixed-effects models using
  {$\ell_1$}-penalization.
\newblock {\em Scand. J. Stat.}, 38(2):197--214, 2011.

\bibitem{She10}
Y.~She.
\newblock Sparse regression with exact clustering.
\newblock {\em Electron. J. Stat.}, 4:1055--1096, 2010.

\bibitem{She12}
Y.~She.
\newblock An iterative algorithm for fitting nonconvex penalized generalized
  linear models with grouped predictors.
\newblock {\em Computational Statistics and Data Analysis}, 10:2976--2990,
  2012.

\bibitem{SunZhang12}
T.~Sun and C.~Zhang.
\newblock Scaled sparse linear regression.
\newblock {\em Biometrika}, 99(4):879--898, 2012.

\bibitem{vandeG07}
S.~van~de Geer.
\newblock The deterministic {L}asso.
\newblock {\em 2007 Proc. Amer. Math. Soc.
  [CD-ROM],~see~also~www.stat.math.ethz.ch/\textasciitilde geer/lasso.pdf},
  2007.

\bibitem{Sara09}
S.~van~de Geer and P.~B{\"u}hlmann.
\newblock On the conditions used to prove oracle results for the {L}asso.
\newblock {\em Electron. J. Stat.}, 3:1360--1392, 2009.

\bibitem{Sara11}
S.~van~de Geer, P.~B{\"u}hlmann, and S.~Zhou.
\newblock The adaptive and the thresholded {L}asso for potentially misspecified
  models (and a lower bound for the {L}asso).
\newblock {\em Electron. J. Stat.}, 5:688--749, 2011.

\bibitem{vdGeer11b}
S.~van~de Geer and J.~Lederer.
\newblock The {B}ernstein-{O}rlicz norm and deviation inequalities.
\newblock {\em Probab. Theory Related Fields}, pages 1--26, 2012.

\bibitem{vdGeer11}
S.~van~de Geer and J.~Lederer.
\newblock The {L}asso, correlated design, and improved oracle inequalities.
\newblock {\em IMS Collections}, 9:303--316, 2013.

\bibitem{Weihuang10}
F.~Wei and J.~Huang.
\newblock Consistent group selection in high-dimensional linear regression.
\newblock {\em Bernoulli}, 16(4):1369--1384, 2010.

\bibitem{Yuan06}
M.~Yuan and Y.~Lin.
\newblock Model selection and estimation in regression with grouped variables.
\newblock {\em J. R. Stat. Soc. Ser. B Stat. Methodol.}, 68(1):49--67, 2006.

\bibitem{BiYuConsistLasso}
P.~Zhao and B.~Yu.
\newblock On model selection consistency of {L}asso.
\newblock {\em J. Mach. Learn. Res.}, 7:2541--2563, 2006.

\bibitem{Zou06}
H.~Zou.
\newblock The adaptive lasso and its oracle properties.
\newblock {\em J. Amer. Statist. Assoc.}, 101(476):1418--1429, 2006.

\end{thebibliography}

\end{document}